\documentclass[11pt]{amsart}
\usepackage{amsthm,amssymb, amsmath}

\usepackage{graphicx, float, epsfig}
\usepackage{tikz}
\usepackage[all, cmtip]{xy}

\usetikzlibrary{arrows}
\usetikzlibrary{snakes}

\DeclareMathOperator{\Tr}{Tr}

\newtheorem{theorem}{Theorem}[section]

\newtheorem{lemma}[theorem]{Lemma}
\newtheorem{proposition}[theorem]{Proposition}

\newtheorem{definition}[theorem]{Definition}

\title{Continuity of core entropy of quadratic polynomials}
\author{Giulio Tiozzo}

\begin{document}

\begin{abstract}The core entropy of polynomials, recently introduced by W. Thurston, 
is a dynamical invariant which can be defined purely in combinatorial terms, and 
provides a useful tool to study parameter spaces of polynomials.
The theory of core entropy extends to complex polynomials the entropy theory 
for real unimodal maps: the real segment is replaced by an invariant tree, known as Hubbard tree, 
which lives inside the filled Julia set. We prove that the core entropy of quadratic polynomials
varies continuously as a function of the external angle, answering a question of Thurston.
\end{abstract}

\maketitle

\section{Introduction}

Recall a polynomial $f : \mathbb{C} \to \mathbb{C}$ is \emph{postcritically finite} if the forward orbits of its critical points are finite. Then the filled Julia set of $f$ contains a forward invariant, finite topological tree, called the \emph{Hubbard tree} \cite{DH}. 

\begin{definition}
The \emph{core entropy} of $f$ is the topological entropy 
of the restriction of $f$ to its Hubbard tree.
\end{definition}

We shall restrict ourselves to quadratic polynomials. 
Given $\theta \in \mathbb{Q}/\mathbb{Z}$, the external ray at angle $\theta$ determines a postcritically finite parameter $c_\theta$ in 
the Mandelbrot set \cite{DH}. We define $h(\theta)$ to be the core entropy of $f_\theta(z) := z^2 + c_\theta$. The main goal of this paper is to prove the following result:

\begin{theorem} \label{T:main}
The core entropy function $h : \mathbb{Q}/\mathbb{Z} \to \mathbb{R}$ 
extends to a continuous function from $\mathbb{R}/\mathbb{Z}$ to $\mathbb{R}$.
\end{theorem}

The theorem answers a question of W. Thurston, who first introduced and explored the core entropy of polynomials.
As Thurston showed, the core entropy function can be defined purely combinatorially, but it displays a rich fractal structure (Figure \ref{F:core}), which reflects the underlying geometry of the Mandelbrot set.

\begin{figure}[h!]
 \label{F:core}
\fbox{
\includegraphics[width = 0.8 \textwidth]{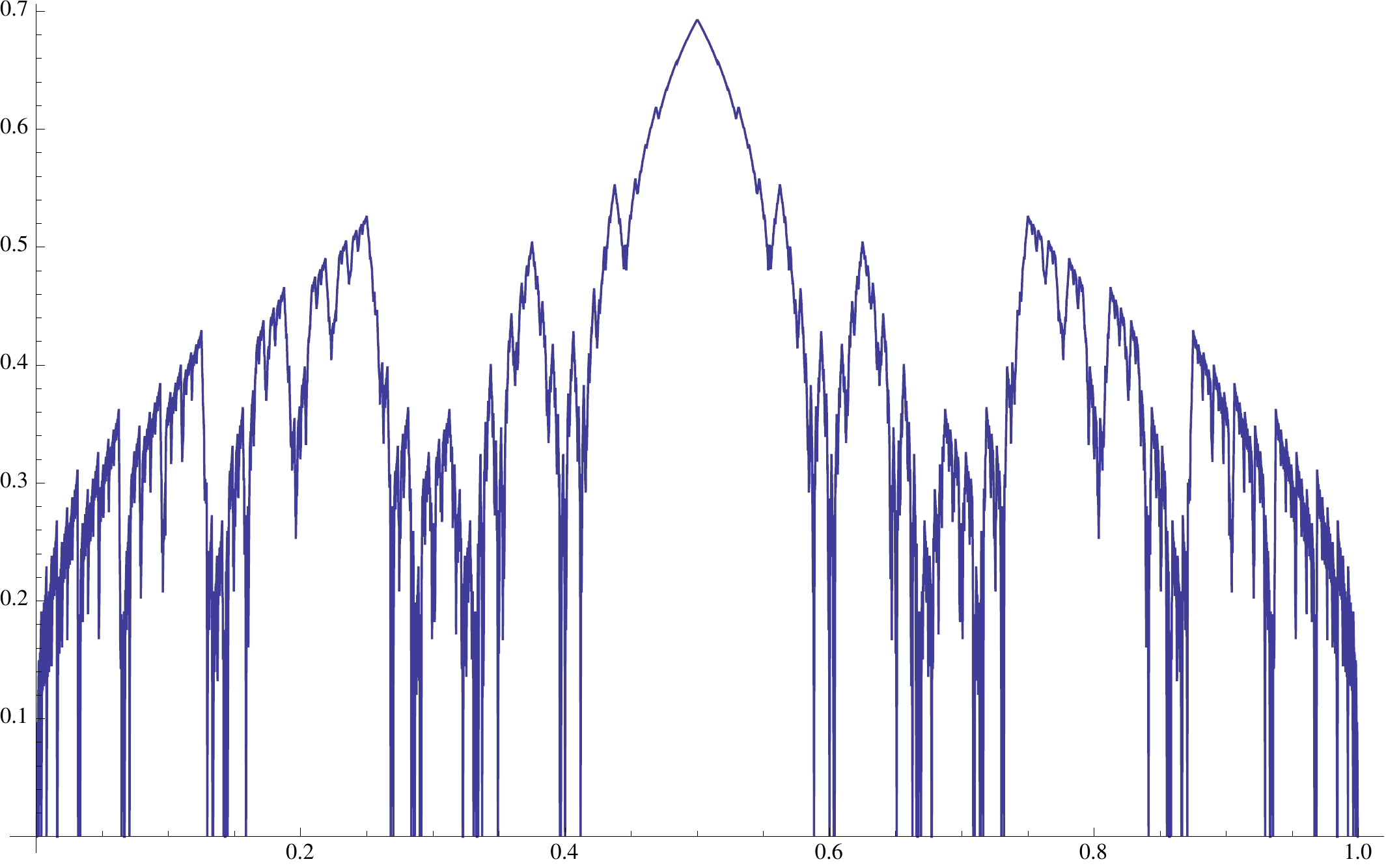}}
\caption{The core entropy $h(\theta)$ of quadratic polynomials as a function of the external angle $\theta$.}
\end{figure}

The concept of core entropy generalizes to complex polynomials the entropy theory 
of real quadratic maps, whose monotonicity and continuity go back to Milnor and Thurston \cite{MT}:
indeed, the invariant real segment is replaced by the invariant tree, which captures all 
the essential dynamics.
On the topological side, Hubbard trees have been introduced to classify postcritically finite maps \cite{DH}, \cite{Po}, 
and their entropy provides a new tool to study the parameter space of polynomials: for instance, 
the restriction of $h(\theta)$ to non-real veins of the Mandelbrot set is 
also monotone \cite{Li}, which implies that the lamination for the Mandelbrot set can be reconstructed 
by looking at level sets of $h$. 
Note by comparison that the entropy of $f_\theta$ on its Julia set is constant, 
independent of $\theta$, hence it does not give information on the parameter.

Furthermore, the value $h(\theta)$ equals, up to a constant factor, the Hausdorff dimension of the set of biaccessible angles for $f_\theta$ (\cite{BS},  \cite{Ti}). For dendritic Julia sets, the core entropy also equals the asymptotic stretch factor of $f_\theta$ as a rational map
\cite{Th1}.




\subsection{Strategy of proof}
A rational external angle $\theta$ determines a postcritically finite map $f_\theta$, which has a finite Hubbard tree $T_\theta$. 
The simplest way to compute its entropy is to compute the Markov transition matrix for the map $f_\theta$ acting on 
$T_\theta$, and take (the logarithm of) its leading eigenvalue: however, this requires to know the topology of $T_\theta$, which 
changes wildly after small perturbations of $\theta$.

In this paper, we shall by-pass this issue by leveraging an algorithm, devised by Thurston, 
which considers instead a larger matrix, whose entries are \emph{pairs} of postcritical angles (see Section \ref{S:algo}); no knowledge 
of the topology of the Hubbard tree is required.
In order to prove $h$ is continuous, we develop an infinite version of such algorithm, which is defined for \emph{any} angle $\theta \in \mathbb{R}/\mathbb{Z}$.
In particular, instead of taking the leading eigenvalue of a transition matrix, we shall encode the possible transitions in a directed graph, 
which will now have countably many vertices. 
By taking the \emph{spectral determinant} of such graph (see Section \ref{S:spectral} for a definition), for each angle $\theta \in \mathbb{R}/\mathbb{Z}$,  we construct a power series $P_\theta(t)$ which converges in the unit disk and 
such that its ``smallest zero" $s(\theta)$ is related to the core entropy by 
$$h(\theta) = - \log s(\theta).$$
We then produce an algorithm to compute each coefficient of the Taylor expansion of $P_\theta(t)$, 
and show that essentially these coefficients vary continuously with $\theta$: the result follows by Rouch\'e's theorem.

As a corollary of our method, we shall prove that the entropy function is actually H\"older continuous at angles $\theta$ such that $h(\theta) > 0$ (using renormalization, it can be proven that $h$ is not H\"older continuous where $h(\theta) = 0$).

\medskip

\subsection{Growth rate of graphs with bounded cycles}
On the way to our proof, we shall develop a few general combinatorial tools to deal with growth rates of countable graphs, 
which may be of independent interest. 
We define a countable graph $\Gamma$ with bounded (outgoing) degree to have \emph{bounded cycles} if it has finitely many closed paths of any given length. In this case, we define the \emph{growth rate} of $\Gamma$ 
as the growth rate with $n$ of the number of closed paths of length $n$ (see Definition \ref{d:growth}); as it turns out (Theorem \ref{T:rootofP}), the inverse of the growth rate equals the smallest zero of the following function $P(t)$, constructed by counting the multi-cycles in the graph:
$$P(t) := \sum_{\gamma \textup{ multi-cycle}} (-1)^{C(\gamma)}\  t^{\ell(\gamma)},$$
where $C(\gamma)$ is the number of components of the multi-cycle $\gamma$, and $\ell(\gamma)$ is its length.
(For the definition of multi-cycle, see Section \ref{s:graphs}; note that our graphs have finite outgoing degree but possibly infinite ingoing degree, hence the 
adjacency operator has infinite $\ell^2$-norm and the usual spectral theory (see e.g. \cite{MW}) does not apply).

We then define a general combinatorial object, called \emph{labeled wedge}, which consists of pairs 
of integers which can be labeled either as being \emph{separated} (representing two elements of the postcritical 
set which lie on opposite sides of the critical point) or \emph{non-separated}; to such object 
we associate an infinite graph, and prove that it has bounded cycles, hence one can apply the theory developed in the first part 
(see Theorem \ref{T:zero}). 

\medskip

\subsection{Application to core entropy}
Finally, we shall apply these combinatorial techniques to the core entropy (Section \ref{S:core}); indeed, we associate to any external angle $\theta$ a 
labeled wedge $\mathcal{W}_\theta$, hence an infinite graph $\Gamma_\theta$, and verify that:
\begin{enumerate}
\item
the growth rate of $\Gamma_\theta$ 
varies continuously as a function of $\theta$ (Theorem \ref{T:newmain}); 
\item
the growth rate of $\Gamma_\theta$ coincides with the core entropy for rational angles (Theorem \ref{T:coincide}).
\end{enumerate}
Let us remark that, as a consequence of monotonicity along veins and Theorem \ref{T:main}, 
the continuous extension we define also coincides with the core entropy for parameters which are not necessarily 
postcritically finite, but for which the Julia set is locally connected and the Hubbard tree is topologically finite.

\medskip 

\subsection{Remarks}
The core entropy of polynomials has been introduced by W. Thurston around 2011 (even though there are earlier related results, e.g. \cite{Pe}, \cite{AF}, \cite{Li}) but most of the theory is yet unpublished (with the exception of Section 6 in \cite{Th2}).
Several people 
are now collecting his writings and correspondence into a foundational paper 
\cite{people}.
In particular, the validity of Thurston's algorithm has been proven by Tan L., Gao Y., and W. Jung (see \cite{Ga}, \cite{Ju}).
Continuity of the core entropy along principal veins in the Mandelbrot set is proven by the author in \cite{Ti}, and 
along all veins by W. Jung \cite{Ju}. 
Note that the previous methods used for veins do not easily generalize, since the topology of the tree is 
constant along veins but not globally.
After \cite{MT}, alternative proofs of monotonicity and continuity of entropy for real maps  are given in \cite{Do}, \cite{Ts}
(the present proof independently yields continuity).
Biaccessible external angles and their dimension have been discussed in 
\cite{Za}, \cite{Zd}, \cite{Sm}, \cite{BS}, \cite{MS}.
 

The method we use to count closed paths in the graph bears many similarities with the theory of \emph{dynamical zeta functions} \cite{AM}, and several forms of the spectral determinant are used in thermodynamic formalism (see e.g. \cite{Ru}, \cite{PP}). 
Moreover, the spectral determinant $P(t)$ we use is an infinite version of the \emph{clique polynomial} 
used in \cite{Mc} to study finite directed graphs with small entropy.


\subsection{The algorithm} \label{S:algo}

The motivation for our combinatorial construction is Thurston's algorithm to compute the core entropy for rational angles, 
which we shall now describe.

Let $\theta \in \mathbb{Q}/\mathbb{Z}$ a rational number $\mod 1$. Then the external ray at angle $\theta$ in the Mandelbrot 
set lands at a Misiurewicz parameter, or at the root of a hyperbolic component: let $f$ denote the corresponding postcritically finite quadratic map.
We shall call $c_0$ the critical point of $f$, and for each $i \geq 1$, $c_i := f^i(c_0)$ the $i^{th}$ iterate of the critical point.
Recall that the \emph{Hubbard tree} of $f$ is the union of the regulated arcs $[c_i, c_j]$ for all $i, j \geq 0$
(for more details, see \cite{DH}).
In the postcritically finite case, it is a finite topological tree, 
and it is forward-invariant under $f$  (see Figure \ref{F:tree}).

\begin{figure} 
\fbox{\includegraphics[width= 0.8\textwidth]{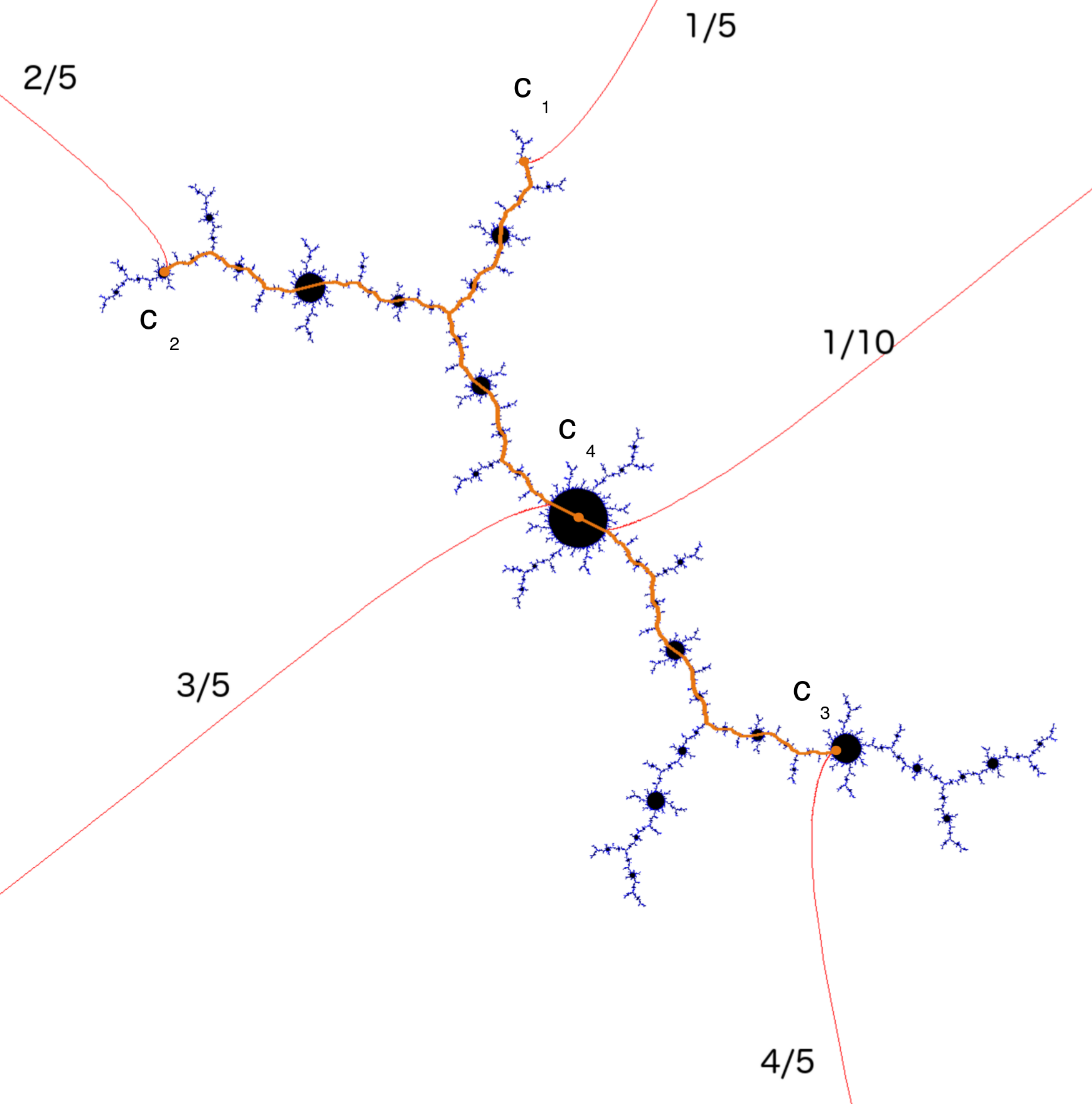}}
\caption{The Hubbard tree for the postcritically finite quadratic polynomial with characteristic angle $\theta = 1/5$.
In this case, the critical point has period $4$, and the tree is a tripod which is the union of the two regulated arcs $[c_1, c_3]$ and $[c_2, c_3]$.}
\label{F:tree}
\end{figure}

It is possible to compute the core entropy of $f$ by writing the Markov transition matrix for the action 
of $f$ on the tree, and take the logarithm of its leading eigenvalue. However, 
given the external angle $\theta$ it is quite complicated to figure out the topology of the tree: 
the following algorithm by-passes this issue by looking at pairs of external angles.

In order to explain the algorithm in more detail, let us remark that a rational angle $\theta$ is eventually periodic under 
the doubling map $ g(x) :=  2x \mod 1$;  that is, there exist integers $p \geq 1$ and $q \geq 0$ such that the elements of the set 
$\mathcal{O} := \{\theta, 2\theta, \dots, 2^{p+q-1}\theta\}$ are all distinct modulo $1$, and $2^{p+q} \theta \equiv 2^q \theta \mod 1$. 
The elements of $\mathcal{O}$ will be called \emph{postcritical angles}; the number $p$
is called the \emph{period} of $\theta$, and $q$ is the \emph{pre-period}. If $q = 0$, we shall call $\theta$ \emph{purely periodic}.
Denote $P_\theta$ the partition of the circle $S^1= \mathbb{R}/\mathbb{Z}$ in the two intervals 
$$ S^1 = \left[\frac{\theta}{2}, \frac{\theta + 1}{2} \right) \cup \left[ \frac{\theta+1}{2} , \frac{\theta}{2} \right). $$
Moreover, for each $i$ let us denote $x_i(\theta) := 2^{i-1} \theta \mod 1$, which we see as a point in $S^1$.

Let us now construct a matrix $M_\theta$ which will be used to compute the core entropy.
Denote $V_\mathcal{O}$ the set of unordered pairs of (distinct) elements of  $\mathcal{O}$: this is a finite set, 
any element of which will be denoted $\{x_i, x_j\}$. We now define a linear map $M_\theta : \mathbb{R}^{V_\mathcal{O}} \to \mathbb{R}^{V_\mathcal{O}}$ by defining it on its basis vectors in the following way:
\begin{itemize}
\item 
if $x_i$ and $x_j$ belong to the same element of the partition $P_\theta$, or at least one of them lies on the boundary,
we shall say that the pair $\{x_i, x_j\}$ is \emph{non-separated}, and define
$$M_\theta(\{x_i, x_j\}) := \{x_{i+1}, x_{j+1}\};$$
\item
if $x_i$ and $x_j$ belong to the interiors of two different elements of $P_\theta$, then we 
say that $\{x_i, x_j\}$ is \emph{separated}, and
define 
$$M_\theta(\{x_i, x_j\}) := \{x_1, x_{i+1}\} + \{x_1, x_{j+1}\}$$
(in order to make sure the formulas are defined in all cases, we shall set $\{x, y\} = 0$ whenever $x = y$).
\end{itemize}

By abuse of notation, we shall also denote $M_\theta$ the matrix representing the linear map $M_\theta$ in the 
standard basis. As suggested by Thurston in his correspondence, the leading eigenvalue of
 $M_\theta$ gives the core entropy:

\begin{theorem}[Thurston; Tan-Gao; Jung] \label{T:algo}
Let $\theta$ a rational angle.
Then the core entropy $h(\theta)$ of the quadratic polynomial of external angle $\theta$ 
is related to the largest real eigenvalue $\lambda$ of the matrix $M_\theta$ 
by the formula
$$h(\theta) = \log \lambda.$$
\end{theorem}

The explanation of this algorithm is the following.
Let $G$ be a complete graph whose vertices are labeled by elements of $\mathcal{O}$, 
and which we shall consider as a topological space.
(More concretely, we can take the unit disk and draw segments between all possible pairs of postcritical angles
on the unit circle: the union of all such segments is a model for $G$ (see Figure \ref{F:chords_algo})).

\begin{figure}\fbox{\includegraphics[width = 0.6\textwidth]{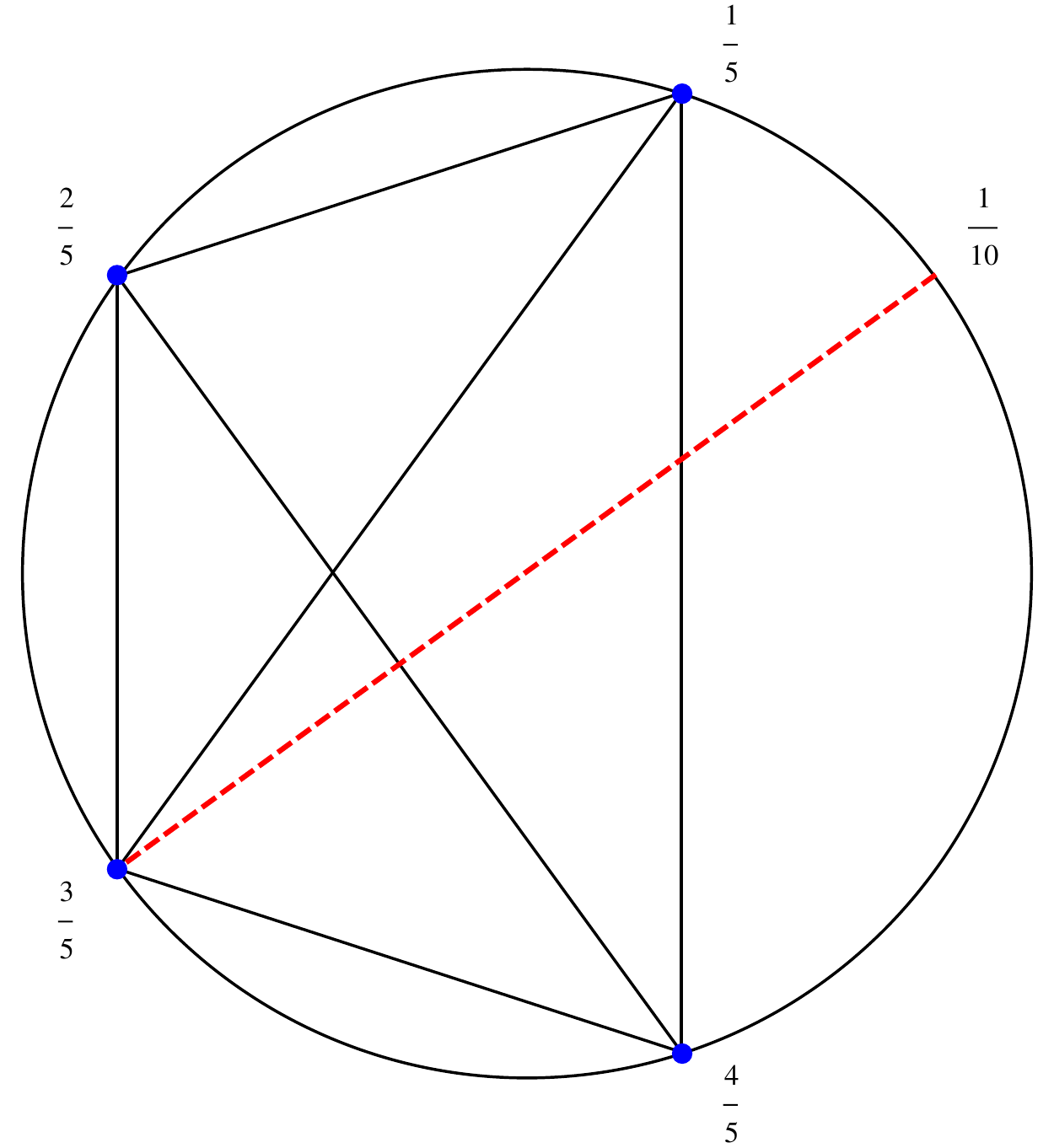}}
\caption{The graph $G$ used to compute the entropy, for $\theta = 1/5$. The vertices of $G$ are the blue points, which correspond to postcritical angles, and the edges of $G$ are the black chords. The dashed red line connects $\theta/2$ to $(\theta+1)/2$ and determines the partition $P_\theta$: pairs of postcritical angles are separated if and only if the chord joining them crosses the dashed line.
For instance, the chord $\{1/5, 2/5\}$ is non-separated, and its image is the chord $\{2/5, 4/5\}$. On the other hand, 
the chord $\{1/5, 4/5\}$ is separated, hence its image is the union of $\{2/5, 1/5\}$ and $\{1/5, 3/5\}$. }
 \label{F:chords_algo}

\end{figure}

Let us denote by $T$ the Hubbard tree of $f$.
We can define a continuous map $\Phi : G \to T$ which sends the edge with vertices $x_i$, $x_j$
homeomorphically to the regulated arc $[c_i, c_j]$ in $T$ (except in the case $c_i = c_j$, where we map all the edge 
to a single point).
Finally, we can lift the dynamics $f : T \to T$ to a map $L : G \to G$ such that $\Phi \circ L = f \circ \Phi$.
In order to do so, 
note that: 
\begin{itemize}
\item
if the pair of angles $\{x_i, x_j\}$ is non-separated, then $f$ maps the arc $[c_i, c_j]$ homeomorphically onto $[c_{i+1}, c_{j+1}]$; 
hence to define $L$ we have to lift $f$ so that it maps the edge $[x_i, x_j]$ homeomorphically onto $[x_{i+1}, x_{j+1}]$;
\item
if instead $\{x_i, x_j\}$ is separated, then the critical point $c_0$ lies on the regulated arc $[c_i, c_j]$ in the Hubbard tree; thus, 
$f$ maps the arc $[c_i , c_j] = [c_i, c_0] \cup [c_0, c_j]$ onto $[c_{i+1}, c_1] \cup [c_1, c_{j+1}]$; so we define $L$ by lifting $f$ and 
so that it maps $[x_i, x_j]$ continuously onto the union $[x_{i+1}, x_1] \cup [x_1, x_{j+1}]$.
\end{itemize}

The map $L$ is a Markov map of a topological graph, hence its entropy is the logarithm of the leading eigenvalue
of its transition matrix, which is by construction the matrix $M_\theta$.
In case $\theta$ is purely periodic, one can prove that the map $\Phi$ is surjective and finite-to-one, and semiconjugates the dynamics $L$ on $G$ to the dynamics $f$ on $T$, hence 
$$h(f, T)= h(L, G) = \log \lambda.$$
More care is needed in the pre-periodic case, since $\Phi$ can collapse arcs to points (for details, see \cite{Ga}, \cite{Ju}, or 
\cite{people}).

\section{Graphs} \label{s:graphs}
 
In the following, by \emph{graph} we mean a directed graph $\Gamma$, i.e. a set $V$ of vertices (which will be finite or countable) 
and a set $E$ of edges, such that each edge $e$ has a well-defined \emph{source} $s(e) \in V$ and a \emph{target} $t(e) \in V$
 (thus, we allow edges between a vertex and itself, 
and multiple edges between two vertices). 
Given a vertex $v$, the set $Out(v)$ of its \emph{outgoing edges} is the set of edges with source $v$.
The \emph{outgoing degree} of $v$ is the cardinality of $Out(v)$; a graph is \emph{locally finite} if the outgoing degree of all its vertices is finite, and has 
\emph{bounded outgoing degree} if there is a uniform upper bound $d < \infty$ on the outgoing degree of all its vertices.
Note that we do \emph{not} require that the ingoing degree is finite, and indeed in our application we will 
encounter graphs with vertices having countably many ingoing edges.

We denote as $V(\Gamma)$ the set of vertices of $\Gamma$, and as $E(\Gamma)$ its set of edges.
Moreover, we denote as $\#(v \to w)$ the number of edges from vertex $v$ to vertex $w$.

A \emph{path} in the graph based at a vertex $v$ is a sequence $(e_1, \dots, e_n)$ of edges such that $s(e_1) = v$, and $t(e_i) = s(e_{i+1})$ 
for $1 \leq i \leq n-1$. The \emph{length} of the path is the number $n$ of edges, and the set of vertices 
$\{s(e_1), \dots, s(e_n) \} \cup \{ t(e_n) \}$ visited by the path is called its \emph{vertex-suppoert}, or just \emph{support}
for simplicity.
Similarly, a \emph{closed path} based at $v$ is a path  $(e_1, \dots, e_n)$ such that $t(e_n) = s(e_1) = v$. 
Note that in this definition a closed path can intersect itself, i.e. two of the sources of the $e_i$ can be the same; 
moreover, closed paths with different starting vertices will be considered to be different.

On the other hand, a \emph{simple cycle} is a closed path which does not self-intersect, modulo cyclical equivalence:
that is, a simple cycle is a closed path $(e_1, \dots, e_n)$ such that $s(e_i) \neq s(e_j)$ for $i \neq j$, and two such paths are considered the same simple cycle if the edges are cyclically permuted, i.e. $(e_1, \dots, e_n)$ and $(e_{k+1}, \dots, e_n, e_1, \dots, e_k)$ 
designate the same simple cycle.
Finally, a \emph{multi-cycle} is the union of finitely many simple cycles with pairwise disjoint (vertex-)supports. The length of a multi-cycle 
is the sum of the lengths of its components.


\medskip

Given a countable graph with bounded outgoing degree, we define the \emph{adjacency operator} $A : \ell^1(V) \to \ell^1(V)$ 
on the space of summable sequences indexed by the vertex set $V$. In fact, 
for each vertex $i \in V$ we can consider the 
sequence $e_i \in \ell^1(V)$ which is $1$ at position $i$ and $0$ otherwise, and define for each $j \in V$ the 
$j^{th}$ component of the vector $A e_i$ to be 
$$(A e_i)_j :=  \#(i \to j)$$
equal to the number of edges from $i$ to $j$.
Since the graph has bounded outgoing degree, the above definition can actually be extended to all $\ell^1(V)$, 
and the operator norm of $A$ induced by the $\ell^1$-norm is bounded above by the outgoing degree $d$ of the graph.
Note moreover that 
for each pair $i, j$ of vertices and each $n$, the coefficient
$(A^n e_i)_j$ equals the number of paths of length $n$ from $i$ to $j$.

\begin{definition}
We say a countable graph $\Gamma$ has \emph{bounded cycles} if it has bounded outgoing degree and for each positive integer $n$, 
$\Gamma$ has at most finitely many simple cycles of length $n$.
\end{definition}

Note that, if $\Gamma$ has bounded cycles, then for each $n$ it 
has also a finite number of closed paths of length $n$, 
since the support of any closed path of length $n$ is contained in the union of the supports of simple cycles with length $\leq n$. 
Thus, for such graphs we shall denote $$C(\Gamma, n)$$ 
the number of closed paths of length $n$. Note that in this case the trace $\Tr(A^n)$ is also well-defined for each $n$, 
and equal to $C(\Gamma, n)$.

\begin{definition} \label{d:growth}
If $\Gamma$ is a graph with bounded cycles, we define the \emph{growth rate} $r = r(\Gamma)$ as the exponential growth rate of the number of its closed paths: 
that is, 
$$r := \limsup_{n} \sqrt[n]{C(\Gamma, n)}.$$
\end{definition}

\subsection{The spectral determinant} \label{S:spectral}

If $\Gamma$ is a finite graph with adjacency matrix $A$, then is well-defined its characteristic polynomial $Q(t) := \det(tI - A)$,
whose roots are the eigenvalues of $A$. In the following we shall work with the related polynomial 
$$P(t) := \det (I - tA)$$
which we call the \emph{spectral determinant} of $\Gamma$. We shall now extend the theory to countable graphs with bounded cycles.

It is known that the spectral determinant $P(t)$ of a finite graph $\Gamma$ is related to its multi-cycles by the following formula (see e.g. \cite{CDS}):
\begin{equation} \label{E:pt}
P(t) := \sum_{\gamma \textup{ multi-cycle}} (-1)^{C(\gamma)}\  t^{\ell(\gamma)}
\end{equation} 
where $\ell(\gamma)$ denotes the length of the multi-cycle $\gamma$, while $C(\gamma)$ is the number of connected components of $\gamma$.

If $\Gamma$ is now a (directed) graph with countably many vertices and bounded cycles, then 
the number of multi-cycles of any given length is finite, hence the formula \eqref{E:pt} above 
is still defined as a formal power series.
Note that we include also the empty cycle, which has zero components and zero length, hence $P(t)$ begins with the constant
term $1$.

Now, let $K(\Gamma, n)$ denote the number of multi-cycles of length $n$ in $\Gamma$, and let us define 
$$\sigma := \limsup_{n} \sqrt[n]{K(\Gamma, n)}$$
 its growth rate. 
Then the main result of this section is the following:

\begin{theorem} \label{T:rootofP}
Suppose we have $\sigma \leq 1$; then 
the formula \eqref{E:pt} defines a holomorphic function $P(z)$ in the unit disk $|z| < 1$, and moreover 
the function $P(z)$ is non-zero in the disk 
$|z| < r^{-1}$; if $r > 1$, we also have $P(r^{-1}) = 0$.
\end{theorem}

The proof uses in a crucial way the following combinatorial statement.
\begin{lemma} \label{L:1/Pt}
Let $\Gamma$ be a countable graph with bounded cycles, and $A$ its adjacency operator.
Then we have the equalities of formal power series
$$\frac{1}{P(t)} = \exp \left( \sum_{m = 1}^\infty \frac{ \Tr(A^m) }{m} t^m \right) = \sum_{
\stackrel{N \in \mathbb{N}}{(m_1, \dots, m_N) \in (\mathbb{N}^+)^N}
} \frac{\Tr(A^{m_1})\cdots \Tr(A^{m_N})}{m_1\cdots m_N N! } t^{m_1+\dots+m_N}$$
where $P(t)$ is the spectral determinant. 
\end{lemma}

\begin{proof}
Note that, since $P(t)$ is a power series with constant term $1$, then $1/P(t)$ is indeed a well-defined power series.
The second equality is just obtained by expanding the exponential function in power series.
To prove the first equality, let us first suppose $\Gamma$ is a finite graph with $n$ vertices, and let $\lambda_1, 
\dots, \lambda_n$ be the eigenvalues of its adjacency matrix (counted with algebraic multiplicity). 
Then $P(t) = \prod_{i = 1}^n (1- \lambda_i t)$ hence 
$$\log(1/P(t)) = \sum_{i = 1}^n -\log(1-\lambda_i t)  = \sum_{i =1}^n \sum_{m =1}^\infty  \frac{\lambda_i^m t^m}{m}  =  \sum_{m=1}^\infty  \frac{t^m (\sum_{i=1}^n \lambda_i^m)}{m}$$
hence the claim follows since $\sum_{i=1}^n \lambda_i^m = \Tr(A^m)$.

Now, let $\Gamma$ be infinite with bounded cycles, and let $M \geq 1$.
Note that both sides of the equation depend, modulo $t^{M+1}$, only on multi-cycles of length $\leq M$, 
which by the bounded cycle condition are supported on a finite subgraph $\Gamma_M$. Thus by applying the 
previous proof to $\Gamma_M$ we obtain equality modulo $t^{M+1}$, and since this holds for any $M$ the claim is proven.
\end{proof}

\begin{proof}[Proof of Theorem \ref{T:rootofP}.]

By the root test, the radius of convergence of $\phi(t) := \sum \frac{\Tr(A^m)}{m} t^m$ is $r^{-1}$.
Thus, since the exponential function has infinite radius of convergence, the radius of convergence of $\exp(\phi(t))$ is at least 
$r^{-1}$, and since $\exp(t) = 1 + t + \psi(t)$ where $\psi(t)$ has positive coefficients, then the radius of convergence of $\exp(\phi(t))$
is exactly equal to $r^{-1}$.

Hence, by Lemma \ref{L:1/Pt}, the radius of convergence of $1/P(t)$ around $t = 0$ is also $r^{-1}$. On the other hand, since $\sigma \leq 1$, then by the root test
the power series $P(t)$ converges inside the unit disk, and defines a holomorphic function; thus, $1/P(t)$ is meromorphic for $|t| < 1$, and holomorphic
for $|t| < r^{-1}$, hence $P(t) \neq 0$ if $|t| < r^{-1}$. 
Moreover, if $r >1$, then the radius of convergence of $1/P(t)$
equals the smallest modulus of one of its poles, hence $r^{-1}$ is the smallest modulus of a zero of $P(t)$.
Finally, since $1/P(t)$ has all its Taylor coefficients real and nonnegative, then 
the smallest modulus of its poles must also be a pole, so $P(r^{-1}) = 0$.
\end{proof}

\begin{lemma} \label{L:PerronEig}
If $\Gamma$ is a finite graph, then its growth rate equals the largest real eigenvalue of its adjacency matrix.
\end{lemma}

\begin{proof}
Note that, by the Perron-Frobenius theorem, since the adjacency matrix is non-negative, it has at least one real eigenvalue whose 
modulus is at least as large as the modulus of any other eigenvalue.
Moreover, if $\lambda$ is the largest real eigenvalue, then $\lambda^{-1}$ is the smallest root of the spectral determinant $P(t)$, hence 
the claim follows from Theorem \ref{T:rootofP}.
\end{proof}

\section{Weak covers of graphs}



Let $\Gamma_1, \Gamma_2$ two (locally finite) graphs. 
A \emph{graph map} from $\Gamma_1$ to $\Gamma_2$ 
is a map
$\pi : V(\Gamma_1) \to V(\Gamma_2)$ on the vertex sets 
and a map on edges 
$\pi : E(\Gamma_1) \to E(\Gamma_2)$ which is compatible, in the sense that 
if the edge $e$ connects $v$ to $w$ in $\Gamma_1$, then the edge $\pi(e)$ connects $\pi(v)$ to $\pi(w)$ in $\Gamma_2$.
We shall usually denote such a map as $\pi : \Gamma_1 \to \Gamma_2$. 

A \emph{weak cover} of graphs is a graph map $\pi : \Gamma_1 \to \Gamma_2$ 
such that: 
\begin{itemize}
\item the map $\pi : V(\Gamma_1) \to V(\Gamma_2)$ between the vertex sets is surjective; 
\item the induced map  
$\pi : Out(v) \to Out(\pi(v))$ between outgoing edges is a bijection for each $v \in V(\Gamma_1)$.
\end{itemize}

Note that the map between outgoing edges is defined because $\pi$ is a graph map, and $\pi$ also induces a map from paths in $\Gamma_1$ 
to paths in $\Gamma_2$. As a consequence of the definition of weak cover, 
you have the following \emph{unique path lifting} property: 

\begin{lemma}
Let $\pi : \Gamma_1 \to \Gamma_2$ a weak cover of graphs. Then given $v \in V(\Gamma_1)$ and $w = \pi(v) \in V(\Gamma_2)$, 
for every path $\gamma$ in $\Gamma_2$ based at $w$ there is a unique path $\widetilde{\gamma}$ in $\Gamma_1$ based at $v$
such that $\pi(\widetilde{\gamma}) = \gamma$.
\end{lemma}

\begin{proof}
Let $v \in V(\Gamma_1)$ and $w = \pi(v) \in V(\Gamma_2)$, and let $\gamma = (e_1, \dots, e_n)$ be a path in $\Gamma_2$ based at $w$.
Since the map $Out(v) \to Out(w)$ is a bijection, there exists exactly one edge $f_1 \in Out(v)$ such that $\pi(f_1) = e_1$. By compatibility, 
the target $u$ of $f_1$ projects to the target of $e_1$, hence we can apply the same reasoning and lift the second edge $e_2$ uniquely starting from $u$,  and so on.
\end{proof}

An immediate consequence of the property is the following

\begin{lemma}  \label{L:easyineq}
Each graph map $\pi : \Gamma_1 \to \Gamma_2$ induces a map 
$$\Pi_{n, v} : \left\{ 
\begin{array}{ll} \textup{closed paths in }\Gamma_1\textup{ of}\\
\textup{length }n\textup{ based at }v 
\end{array} \right\} 
\to 
\left\{ 
\begin{array}{ll} \textup{closed paths in }\Gamma_2\textup{ of}\\
\textup{length }n\textup{ based at }\pi(v) 
\end{array} \right\} 
$$
for each $v \in \Gamma_1$ and each $n \geq 1$. Moreveor, if $\pi$ is a weak cover, then $\Pi_{n, v}$ is injective.
\end{lemma}

\subsection{Quotient graphs}

A general way to construct weak covers of graphs is the following. 
Suppose we have an equivalence relation $\sim$ on the vertex set $V$ of a locally finite graph, and denote $\overline{V}$ 
the set of equivalence classes of vertices. Such an equivalence relation is called \emph{edge-compatible} if whenever $v_1 \sim v_2$, 
for any vertex $w$
the total number of edges from $v_1$ to the members of the equivalence class of $w$ equals the total number of edges from $v_2$ to 
the members of the equivalence class of $w$.
When we have such an equivalence relation, we can define a quotient graph $\overline{\Gamma}$ with vertex set $\overline{V}$. Namely, we denote for each $v, w \in V$ the respective 
equivalence classes as $[v]$ and $[w]$, and define the number of edges from $[v]$ to $[w]$ in the quotient graph to be 
$$\#([v] \to [w]) := \sum_{u \in [w]} \#(v \to u).$$
By definition of edge-compatibility, the above sum does not depend on the representative $v$ chosen inside the class $[v]$. Moreover, 
it is easy to see that the quotient map 
$$\pi : \Gamma \to \overline{\Gamma}$$
is a weak cover of graphs.

Let us now relate the growth of a graph to the growth of its weak covers.

\begin{lemma} \label{L:growthcover}
Let $\pi : \Gamma_1 \to \Gamma_2$ a weak cover of graphs with bounded cycles, and $S$ a finite set of vertices of $\Gamma_1$.
\begin{enumerate}
\item Suppose that every closed path in $\Gamma_1$ 
passes through $S$. Then for each $n$ we have the estimate
$$C(\Gamma_1, n) \leq n \cdot \# S \cdot C(\Gamma_2, n)$$
which implies 
$$r(\Gamma_1) \leq r(\Gamma_2).$$
\item
Suppose that $\mathcal{G}$ is a set of closed paths in $\Gamma_2$ such that each $\gamma \in \mathcal{G}$
crosses at least one vertex $w$ with the property that: the set $S_w := \pi^{-1}(w) \cap S$ is non-empty, and
any lift of $\gamma$
from an element of $S_w$ ends in $S_w$. Then there exists $L \geq 0$, which depends on $S$, such that for each $n$ we have 
$$\#\{ \gamma \in \mathcal{G} \ : \ \textup{length}(\gamma) = n \} \leq n \sum_{k = 0}^L C(\Gamma_1, n+k).$$
\end{enumerate}
\end{lemma}

\begin{proof}
(1) Let $\gamma$ be a closed path in $\Gamma_1$ of length $n$, based at $v$. Let now $v_0$ be the first vertex of $\gamma$ which belongs to $S$ (by hypothesis, 
there is one). If we call $\gamma_0$ the cyclical permutation of $\gamma$ based at $v_0$, the projection of $\gamma_0$ 
to $\Gamma_2$ now yields a closed path $\gamma'$ in $\Gamma_2$ which is based at $v' = \pi(v_0) \in \pi(S)$. Now
note that for each such pair $(\gamma', v')$, there are at most $\#S$ possible choices for $v_0 \in \pi^{-1}(v')\in S$; then, given $v_0$ there is a unique lift of $\gamma'$
from $v_0$, and $n$ possible choices for the vertex $v$ on that lift, giving the estimate.

(2)
Recall first that in every directed graph we can define an equivalence relation by saying $v \sim w$ 
if there is a path from $v$ to $w$ and a path from $w$ to $v$. The equivalence classes are known 
as \emph{strongly connected components}, or $s.c.c.$ for short. Moreover, we can form a graph $\Gamma_{scc}$ whose vertices are 
the strongly connected components and there is an edge from the s.c.c. $K_1$ to $K_2$ if in the original graph 
there is a path from an element of $K_1$ to an element of $K_2$. Note that by construction 
this graph has no cycles.

Moreover, let us consider a pair $(v_1, v_2)$ of distinct elements of $S$; 
then either there is no path from $v_1$ to $v_2$, or we can pick some path from $v_1$ to $v_2$, which will be denoted 
$\gamma(v_1, v_2)$. Then the set $\mathcal{P} := \{ \gamma(v_1, v_2) \ :  \ v_1, v_2 \in S \}$ of such paths 
is finite, and let $L$ be the maximum length of an element of $\mathcal{P}$.

Let now $\gamma$ be a closed path in $\Gamma_2$ based at $w_0$. By hypothesis, $\gamma$ passes through some $w$
such that $S_w = \pi^{-1}(w) \cap S \neq \emptyset$, and each lift of $\gamma$ from $S_w$ ends in $S_w$. 
Let us now consider the set $\mathcal{K}_w$ of all s.c.c. of $\Gamma_1$ which intersect $S_w$: since the graph $\Gamma_{scc}$ constructed above has 
no cycles, there is a component $K_0 \in \mathcal{K}_w$ which has the property that there is no path from $K_0$ to 
some other component of $S_w$. Thus, if we pick $v \in K_0$ and lift $\gamma$ from $v$ to a path $\widetilde{\gamma}$ in $\Gamma_1$, the endpoint $u$ 
of $\widetilde{\gamma}$ must lie inside $S_w$ by hypothesis. Thus, by the property of $K_0$, the vertex $u$ lies in $K_0$, hence there is a path $\gamma(u, v)$ from $u$ to $v$ 
in the previously chosen set $\mathcal{P}$, hence if we take $\gamma' = \widetilde{\gamma} \cup \gamma(u, v)$, this is a closed path in $\Gamma_1$ of length between $n$ and $n+L$.
Thus we have a map
$$\left\{  \begin{array}{ll} 
\textup{closed paths of }\mathcal{G}\\
\textup{of length }n 
\end{array} \right\} \to 
\left\{ \begin{array}{ll}
 \textup{closed paths in }\Gamma_1 \\
 \textup{of length }\in \{n, \dots, n + L\} 
 \end{array} \right\}$$
given by $(\gamma, w_0) \to (\gamma', v)$. 
Now, given $v$ we can recover $w = \pi(v)$, and given $\gamma'$ we can recover $\gamma$, since $\gamma$ is the path given by the first $n$ edges of $\pi(\gamma')$ starting at $w$. Finally, we have at most $n$ choices for the 
starting point $w_0$ on $\gamma$, hence the fibers of the above map have cardinality at most $n$, proving the claim.
\end{proof}



An immediate corollary of the Lemma is the following, when $\Gamma_1$ and $\Gamma_2$ are both finite.

\begin{lemma} \label{L:sameentro}
Let $\Gamma_1, \Gamma_2$ be finite graphs, and $\pi : \Gamma_1 \to \Gamma_2$ a weak cover. 
Then the growth rate of $\Gamma_1$ equals the growth rate of $\Gamma_2$.
\end{lemma} 

\begin{proof}
Apply the lemma with $S = V(\Gamma_1)$, and $\mathcal{G}$ equal to the set of all closed paths in $\Gamma_2$.
\end{proof}

\section{Wedges}

Let us consider the set $\Sigma := \{ (i, j)\in \mathbb{N}^2 \ : \ 1 \leq i < j \}$ of pairs of disjoint positive integers. The set $\Sigma$ will be sometimes called 
the \emph{wedge} and, given an element $v = (i, j) \in \Sigma$, 
the coordinate $i$ will be called the \emph{height} of $v$, while the coordinate $j$ will be called the \emph{width} of $v$. The terminology 
becomes more clear by looking at Figure \ref{F:wedge}.

\begin{figure}[h!] 
$$\xymatrix{
& & & & \cdots \\
& & & (4,5) & \cdots \\
& & (3,4)  & (3,5) & \cdots \\
& (2,3) & (2,4) & (2,5) & \cdots \\
(1,2) & (1,3) & (1,4) & (1,5) & \cdots \\
}$$
\caption{The wedge $\Sigma$.}
\label{F:wedge}
\end{figure}
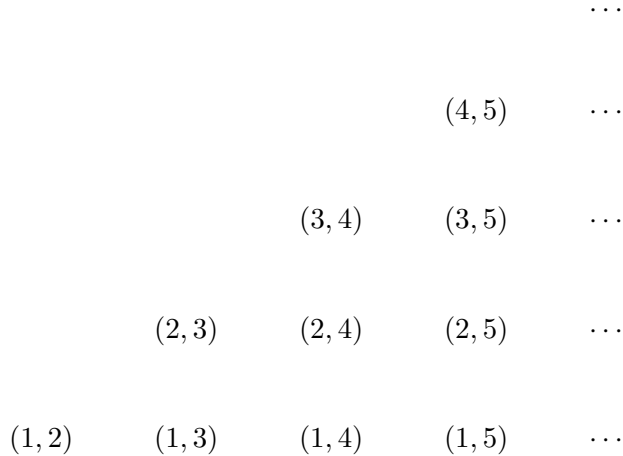

\begin{definition}
We call a \emph{labeled wedge} an assignment $\Phi : \Sigma \to \{N, S\}$ of a label $N$ (which stands for \emph{non-separated}) or $S$
(which stands for \emph{separated}) to each element of $\Sigma$. 
\end{definition}

Now, to each labeled wedge $\mathcal{W}$ we assign a graph $\Gamma_\mathcal{W}$ in the following way.
The vertex set of $\Gamma_\mathcal{W}$ is the wedge $\Sigma$, while the edges of $\Gamma_\mathcal{W}$ are 
labeled by the set $\{U, B, F\}$ 
and determined according to the following rule.
\begin{itemize} 
\item if $(i, j) \in \Sigma$ is non-separated, then $(i,j)$ has as its (unique) successor the vertex $(i+1, j+1)$; 
we say that the edge $(i, j) \to (i+1, j+1)$ is an \emph{upward edge} and we label it with $U$;
\item if $(i, j)$ is separated, then $(i, j)$ has two successors: 
\begin{itemize}
\item 
first, we add the edge $(i, j) \to (1, j+1)$, which we call \emph{forward edge} and label it with $F$.
\item
second, we add the edge $(i, j) \to (1, i+1)$, which we call \emph{backward edge} and label it with $B$.
\end{itemize}
\end{itemize}

In order to explain the names, note that following an upward or forward edge increases the width by $1$, while following a backward edge (weakly) decreases it. 
Moreover, following an upward edge increases the height by $1$, while the targets of both backward and forward edges have height $1$.

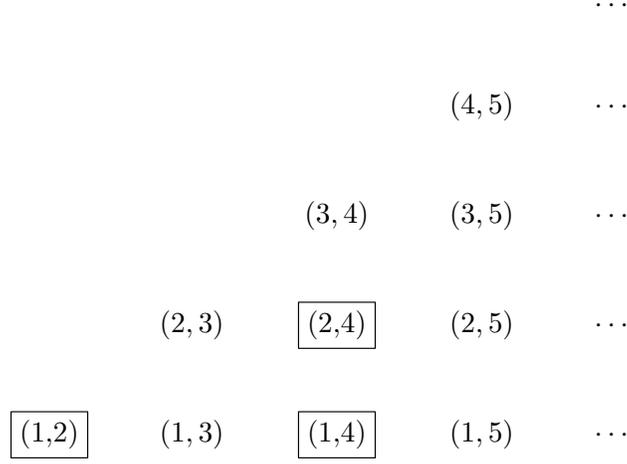
\begin{figure} 
$$\xymatrix{
& & & & \cdots \\
& & & (4,5) & \cdots \\
& & (3,4)  & (3,5) & \cdots \\
& (2,3)   & *+[F]\txt{(2,4)}   & (2,5) & \cdots \\
*+[F]\txt{(1,2)}  & (1,3)   & *+[F]\txt{(1,4)}  & (1,5) & \cdots \\
}$$
\caption{An example of a labeled wedge $\mathcal{W}$. The boxed pairs are the separated ones.}
\label{F:labwedge}
\end{figure}

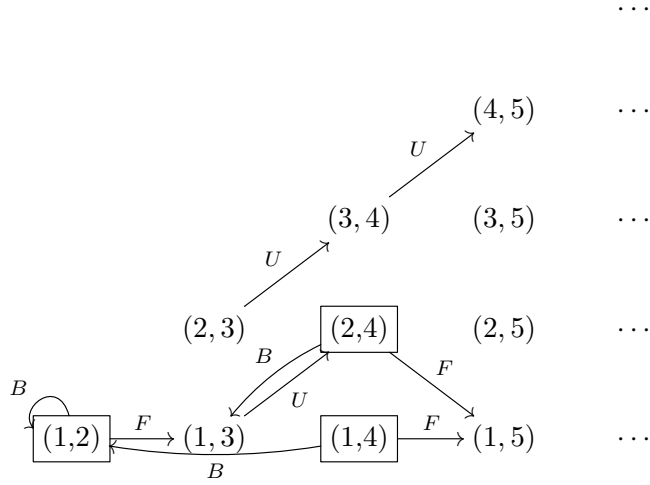
\begin{figure}
$$\xymatrix{
& & & & \cdots \\
& & & (4,5) & \cdots \\
& & (3,4) \ar[ur]^U & (3,5) & \cdots \\
& (2,3)  \ar[ur]^U & *+[F]\txt{(2,4)} \ar[dr]^F  \ar@/_/[dl]_B & (2,5) & \cdots \\
*+[F]\txt{(1,2)} \ar@(u, l)_B  \ar[r]^F & (1,3) \ar[ur]_U  & *+[F]\txt{(1,4)} \ar[r]^F \ar@/^/[ll]^B & (1,5) & \cdots \\
}$$
\caption{The graph $\Gamma_\mathcal{W}$ associated to the labeled wedge $\mathcal{W}$ of Figure \ref{F:labwedge}.}
\end{figure}

\begin{proposition} \label{P:properties}
Let $\Gamma_\mathcal{W}$ the graph associated to the labeled wedge $\mathcal{W}$. Then the following are true:
\begin{enumerate}
\item each vertex of a closed path of length $n$ has height at most $n$; 
\item the support of each closed path of length $n$ intersects the set $\{ (1, k) \ : \ 2 \leq k \leq n+1 \}$; 
\item each vertex of a closed path of length $n$ has width at most $2n$;
\item for each $k \geq 1$, there exists at most one separated vertex in the $k^{th}$ diagonal 
$$D_{k} := \{ (i, j) \in \Sigma \ : \ j-i = k \}$$
which is contained in the support of at least one closed path; 
\item there are at most $(2n)^{\sqrt{2n}}$ multi-cycles of length $n$.
\end{enumerate}
\end{proposition}

Note that (1), (2), (3) are sharp, as seen by the simple cycle 
$$\xymatrix{ 
14 \ar@/^/[r]^U & 25 \ar@/^/[r]^U  & *+[F]\txt{36} \ar@/^/[ll]^B
}$$

\begin{proof}
(1) Let us first note that every closed path contains at least one backward edge, since the upward and forward edges always increase the height.
Moreover, the endpoint of a backward edge has always height $1$, and each edge increases the height by at most $1$, hence the height of a vertex 
along the closed path is at most $n$.

(2) Since the target of each backward edge is $(1, i+1)$, where $i$ is the height of the source of the edge, which is at most $n$ by the previous point, 
then the target of each backward edge along the closed path belongs to the set $\{(1, k) \ : \ 2 \leq k \leq n+1\}$.

(3) By the previous point, there is at least a vertex along the closed path with width at most $n+1$. Since every move increases the width by at most $1$, then 
the largest possible width of a vertex along the path is $(n+1) + (n-1) = 2n$.

(4) Let $v_k =  (h_k, w_k)$ be the separated vertex in $D_k$ with smallest height, if there is one. We claim that no vertex $(h_k + t, w_k + t)$ of $D_k$
with $t \geq 1$ belongs to any closed path. 
Note by looking at the rules that, if a vertex $v = (i, j)$ of height $i >1$ is the target of some edge, then it must be the target of an upward edge, 
more precisely an edge from the vertex $(i-1, j-1)$ immediately to the lower left of $v$, which then must be non-separated.  
Thus, since $v_k$ is separated, the vertex $(h_k+1, w_k+1)$ does not belong to any closed path; the claim then follows by induction on $t$, since, by the same 
reasoning, if $(h_k+t+1, w_k+t+1)$ belongs to some closed path, then also $(h_k+t, w_k+t)$ must belong to the same path.

(5) Let $e_1, \dots, e_r$ the backward edges along a multi-cycle $\gamma$ of length $n$, and denote $v_i$ the source of $e_i$, and 
$h_i$ the height of $v_i$. Note that the set $\{v_1, \dots v_r\}$ determines $\gamma$, as you can start from $v_1$, follow the backward edge, 
and then follow upward or forward edges 
until you either close the loop or encounter another $v_i$, and then continue this way until you walk along all of $\gamma$. Moreover, 
we know that for each $v_i$ there are at most $2n$ possible choices, as each $v_i$ is separated and by (4) there is at most one for each diagonal $D_k$, and 
by (3) it must lie on some $D_k$ with $k \leq 2n$. 
We now claim that 
$$r \leq \sqrt{2n}$$
which is then sufficient to complete the proof.
Let us now prove the claim. By definition of multi-cycle, then the targets of the $e_i$ must be all distinct, and by the rule these targets are precisely $(1, h_i + 1)$ with $1 \leq i \leq r$, 
hence all $h_i$ must be distinct. Moreover, let us note that each $e_i$ must be preceded along the multi-cycle by a sequence 
$$(1, k) \to (2, k + 1) \to \dots \to (h_i, h_i + k - 1)$$
of upward edges of length $h_i - 1$, and all such sequences for distinct $i$ must be disjoint. Hence we have that all $h_i$ are distinct and their total sum 
is at most the length of the multi-cycle, i.e. $n$. Thus we have 
$$\frac{r^2}{2} \leq \frac{r(r+1)}{2}  = \sum_{i = 1}^r i \leq \sum_{i = 1}^r h_i \leq n$$
which proves the claim.
\end{proof}

\begin{theorem} \label{T:zero}
Let $\mathcal{W}$ be a labeled wedge. Then its associated graph $\Gamma$ has bounded cycles, and its spectral determinant 
$P(t)$ defines a holomorphic function in the unit disk. Moreover, the growth rate $r$ of the graph $\Gamma$ equals the inverse of 
the smallest real positive root of $P(z)$, in the following sense:
$P(z) \neq 0$ for $|z| < r^{-1}$ and, if $r > 1$, then $P(r^{-1}) = 0$.
\end{theorem}

\begin{proof}
By construction, the outgoing degree of any vertex of $\Gamma$ is at most $2$. Moreover, by Proposition \ref{P:properties} (5) the graph has bounded cycles, and 
the growth rate of the number $K(\Gamma, n)$ of multi-cycles is $\leq 1$, since 
$$ \limsup_{n} \sqrt[n]{(2n)^{\sqrt{2n}}} = \limsup_{n} e^{\frac{\sqrt{2} \log(2n)}{\sqrt{n} } } = 1.$$
The claim then follows by Theorem \ref{T:rootofP}.
\end{proof}

We shall sometimes denote as $r(\mathcal{W})$ the growth rate of the graph associated to the labeled wedge $\mathcal{W}$.
We say that a sequence $(\mathcal{W}_n)_{n \in \mathbb{N}}$ of labeled wedges converges to $\mathcal{W}$ if for each finite set 
of vertices $S \subseteq \Sigma$ there exists $N$ such for each $n \geq N$ the labels of the elements of $S$ for $\mathcal{W}_n$ and $\mathcal{W}$ 
are the same.

\begin{lemma} \label{L:contwedge}
If a sequence of labeled wedges $(\mathcal{W}_n)_{n \in \mathbb{N}}$ converges to $\mathcal{W}$, then the growth rate of $\mathcal{W}_n$ converges to the growth rate of $\mathcal{W}$.
\end{lemma}

\begin{proof}
Let $P_n(t)$ and $P(t)$ denote respectively the spectral determinants of $\mathcal{W}_n$ and $\mathcal{W}$. Then for each $k \geq 1$, 
the coefficient of $t^k$ in $P_n(t)$ converges to the coefficient of $t^k$ in $P(t)$, because by Proposition \ref{P:properties} (1) and (3) the support of any multi-cycle of length $k$ is contained in the finite subgraph $\{ (i, j) \ : \ 1 \leq i < j \leq 2k \}$. Thus, since the modulus of the coefficient of $t^k$ is uniformly bounded above 
by $(2k)^{\sqrt{2k}}$, then $P_n(t) \to P(t)$ uniformly on compact subsets of the unit disk. Thus, by Rouch\'e's theorem, the smallest real positive zero 
of $P_n(t)$ converges to the smallest real positive zero of $P(t)$, hence by Theorem \ref{T:zero} we have $r(\mathcal{W}_n) \to r(\mathcal{W})$.
\end{proof}

\section{Periodic wedges}

Given integers $p \geq 1$ and $q \geq 0$ we define the equivalence relation $\equiv_{p, q}$ on $\mathbb{N}^+$ by saying that 
$i \equiv_{p, q} j$ if: 
\begin{itemize}
\item either $\min\{ i, j\} \leq q$ and $i = j$; 
\item or $\min\{i, j\} \geq q+1$ and $i \equiv j \mod p$.
\end{itemize} 

Note that if $q = 0$ the equivalence relation $\equiv_{p,q}$ is simply the congruence modulo $p$.
A set of representatives for the equivalence classes of $\equiv_{p,q}$ is the set $\{1, 2, \dots, p+q\}$.
The equivalence relation induces an equivalence relation on the set $\mathbb{N}^+ \times \mathbb{N}^+$ of ordered pairs of integers by saying that
 $(i,j) \equiv_{p,q} (k, l)$ if $i \equiv_{p,q} k$ and $k \equiv_{p,q} l$.
Moreover, it also induces an equivalence relation on the set of \emph{unordered pairs} of integers by saying that the unordered pair 
$\{i, j\}$ is equivalent to $\{k, l\}$ if either $(i, j) \equiv_{p,q} (k, l)$ or $(i, j) \equiv_{p,q} (l,k)$.

\begin{definition}
A labeled wedge is \emph{periodic} of period $p$ and pre-period $q$ if the following two conditions hold:
\begin{itemize}
\item 
any two pairs $(i, j)$ and $(k, l)$ such that $\{i, j\} \equiv_{p,q} \{k, l\}$ have the same label;
\item
if $i \equiv_{p, q} j$, then the pair $(i, j)$ is non-separated.
\end{itemize}
If $q=0$, the labeled wedge will be called \emph{purely periodic}.
\end{definition}

A pair $(i, j)$ with $i \equiv_{p,q} j$ will be called \emph{diagonal}; hence the second point in the definition can be rephrased as 
``every diagonal pair is non-separated".

\begin{figure}[h!] 
$$\xymatrix{
& & & & \dots\\
& & & (4,5)  \ar[ur] &  \dots \\
& & (3,4)  \ar[ur]  & (3, 5) \ar[ur]  & \dots\\
& *+[F]\txt{(2,3)} \ar[d] \ar[dr] & *+[F]\txt{(2,4)} \ar[dl] \ar[dr] & *+[F]\txt{(2, 5)} \ar[dll] \ar[dr] &  \dots \\
(1,2) \ar[ur]  & *+[F]\txt{(1,3)} \ar[r] \ar[l] & *+[F]\txt{(1, 4)} \ar[r] & *+[F]\txt{(1, 5)} \ar[r] &  \dots}$$
\caption{A periodic labeled wedge, of period $p = 1$ and pre-period $q = 2$, and its infinite graph. Note that the vertex $(1,3)$ 
has infinitely many incoming edges.}
\label{F:infinite}
\end{figure}
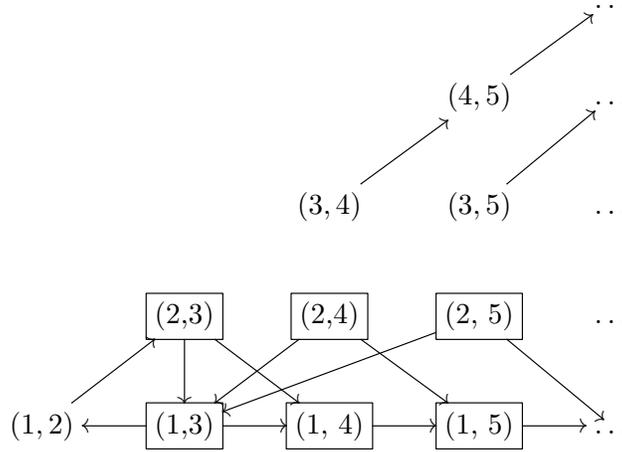

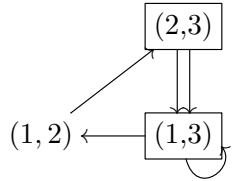
\begin{figure}[h!]
$$\xymatrix{
& *+[F]\txt{(2,3)} \ar@<0.5ex>[d] \ar@<-0.5ex>[d] \\
(1,2) \ar[ur] & *+[F]\txt{(1,3)} \ar[l]   \ar@(d,r)
}$$
\caption{The finite model of the infinite graph associated to the labeled wedge of Figure \ref{F:infinite}, of period $p = 1$ and pre-period $q = 2$.}
\label{F:finite}
\end{figure}

\subsection{The finite model} \label{S:finitemodel}

Given a periodic wedge $\mathcal{W}$ of period $p$ and pre-period $q$, with associated (infinite) graph $\Gamma$, 
we shall now construct a finite graph $\Gamma^F$ which captures
the essential features of the infinite graph $\Gamma$, in particular its growth rate.

The set of vertices of $\Gamma^F$ is the set of $\equiv_{p,q}$-equivalence classes of non-diagonal, 
unordered pairs of integers. A set of representatives of $V(\Gamma^F)$ is 
the set 
$$\Sigma_{p, q} := \{ \{i, j\} \ : \ 1 \leq i < j \leq p+q \}.$$ 
The edges of $\Gamma^F$ are induced by the edges of $\Gamma$, that is are determined by 
the following rules: 
if the unordered pair $\{i, j\}$ is non-separated, then $\{i, j\}$ has one outgoing edge, namely $\{i, j\} \to \{i+1, j+1\}$;
while if $\{i, j\}$ is separated, then $\{i, j\}$ has the two outgoing edges $\{i, j\}\ \to \{1, i+1\}$, 
and $\{i, j\} \to \{1, j+1\}$.

The main result of this section is the following.

\begin{proposition} \label{P:finiteinfinite}
Let $\mathcal{W}$ be a periodic labeled edge, with associated (infinite) graph $\Gamma$. 
Then the growth rate of $\Gamma$ equals the growth rate of its finite model $\Gamma^F$.
\end{proposition}

In order to prove the Proposition, we shall also introduce an intermediate finite graph, which we  
call the \emph{finite 2-cover} of $\Gamma^F$, and denote $\Gamma^{(2)}$.

The set of vertices of $\Gamma^{(2)}$ is the set of $\equiv_{p, q}$-equivalence classes of non-diagonal, ordered pairs of integers, 
and the edges are induced by the edges of $\Gamma$ in the usual way.
The reason to introduce the intermediate graph $\Gamma^{(2)}$ is that $\Gamma^F$ 
does not inherit the labeling of edges from $\Gamma$, as backward and forward edges in $\Gamma$
may map to the same edge in $\Gamma^F$, while $\Gamma^{(2)}$ naturally inherits the labels. 

\begin{proof}[Proof of Proposition \ref{P:finiteinfinite}]
Let $\Gamma$ be the graph associated to a periodic labeled wedge.
\begin{enumerate}

\item
First, let us observe that no diagonal vertex is contained in the support of any closed path of $\Gamma$:
in fact, every diagonal vertex is non-separated, and its outgoing edge leads to another diagonal vertex with 
larger height, hence the path can never close up.
Thus, we can construct the subgraph $\Gamma^{ND}$ by taking as vertices all pairs which are non-diagonal, 
and as edges all the edges of $\Gamma$ which do not have either as a source or target a diagonal pair.
By what has been just said, the growth rate of $\Gamma$ and $\Gamma^{ND}$ is the same,
$$r(\Gamma^{ND}) = r(\Gamma).$$
\item 
Since the maps 
$$\Gamma^{ND} \to \Gamma^{(2)} \to \Gamma^F$$ 
are given by quotienting with respect to equivalence relations, they are both weak covers of graphs.
Thus, since both $\Gamma^{(2)}$ and $\Gamma^F$ are finite, by Lemma \ref{L:sameentro}, 
the growth rates of $\Gamma^{(2)}$ and $\Gamma^F$ are the same. 
\item 
We are now left with proving that the growth rate of $\Gamma^{ND}$ is the same as the growth rate of $\Gamma^{(2)}$.
Since the cardinality of the fiber of the projection $\Gamma^{ND} \to \Gamma^{(2)}$ is infinite, the statement is not immediate.
Note that the finite $2$-cover $\Gamma^{(2)}$ is a graph with labeled edges: indeed, 
if $(i, j) \equiv_{p,q} (k, l)$, then $(1, i+1) \equiv_{p,q} (1, k+1)$, and so on, 
thus the labeling of $\Gamma^{(2)}$ inherited from $\Gamma$ is well-defined, and the graph map 
$\Gamma^{ND} \to \Gamma^{(2)}$ preserves the labels.
\item
Let us call \emph{backtracking} a path in $\Gamma$ or $\Gamma^{(2)}$ such that at least one of its edges is
labeled by $B$ (= backward), and \emph{non-backtracking}
otherwise. Now let us note the following: 
\begin{enumerate}
\item 
Every closed path in $\Gamma$ is backtracking; in fact, following any edge which is upward or forward
increases the height, thus a path in $\Gamma$ made entirely of $U$ and $F$ edges cannot close up.

\item
Every closed path in $\Gamma$ passes through the finite set 
$$S = \{ (1, h) \ : \ 2 \leq h \leq p+q+1 \}.$$ 
In order to prove this, we first prove that any element in the support of any closed path in $\Gamma$ 
has height at most $p + q$.
In order to do so, let us fix a diagonal $D_k = \{(i, j) \ : \ j-i = k \}$.
By periodicity, either there is a separated pair $(i, j)$ in $D_k$ of height $i \leq p+q$, or all 
elements of $D_k$ are non-separated; in the latter case, no element of $D_k$ is part 
of any closed path, since any path based at an element of $D_k$ is non-backtracking. 
In the first case, let $(i_0, j_0)$ be the separated pair with smallest height in $D_k$; 
then, only the elements with height less than $i_0$ can be part of any closed path, and $i_0 \leq i \leq p+q$, 
so the first claim is proven.
As a consequence, the target of any $B$-labeled edge which belongs to some closed path
is of type $(1, h+1)$, where $h \leq p+q$ is the height of the source, hence the target belongs to $S$. 
By the same reasoning, if $\gamma$ is any path in $\Gamma$ based at a vertex of height $1$, 
then the target of any $B$-labeled edge along $\gamma$ lies in $S$.
\item 
Every backtracking closed path in $\Gamma^{(2)}$ has at least one $B$-labeled edge, hence 
the target of such $B$-labeled edge lies in the set $\pi(S)$, and every lift to $\Gamma$ starting from 
an element of $S$ must end in the target of a $B$-labeled edge in $\Gamma$, hence must end in $S$.
\end{enumerate}
\item
Finally, let us note that for each $n$, the number of non-backtracking paths of length $n$ in $\Gamma^{(2)}$ is at most the
cardinality of $V(\Gamma^{(2)})$; indeed, from each vertex there is at most one edge labeled $U$ or $F$, thus for each 
vertex of $\Gamma^{(2)}$ there is at most one non-backtracking path of length $n$ based at it.
\item 
Let us now put together the previous statements. Indeed, by (4)(a)-(b) every closed path in $\Gamma$ passes through $S$, 
hence we can apply Lemma \ref{L:growthcover} (1) and get that 
$$r(\Gamma^{ND}) \leq r(\Gamma^{(2)}).$$
To prove the other inequality, let us note that by (4)(b)-(c) we know that Lemma \ref{L:growthcover} (2) applies with $\mathcal{G}$ the set of backtracking closed paths in $\Gamma^{(2)}$.
Moreover, by point (5) above we have that the number of non-backtracking paths is bounded independently 
of $n$: thus we can write 
$$C(\Gamma^{(2)}, n) = \#\left\{\begin{array}{ll} \textup{backtr. closed}\\ \textup{paths of length }n \end{array} \right\} + 
 \#\left\{\begin{array}{ll} \textup{non-backtr. closed}\\ \textup{paths of length }n \end{array} \right\}\leq $$
$$\leq n \sum_{k = 0}^L C(\Gamma^{ND}, n+k) + \# V(\Gamma^{(2)})$$
from which follows
$$r(\Gamma^{(2)}) \leq r(\Gamma^{ND})$$ 
as required.
\end{enumerate}

\end{proof}

When dealing with purely periodic external angles, we shall also need the following lemma.

\begin{lemma} \label{L:twoperiodics}
Let $\mathcal{W}_1$ and $\mathcal{W}_2$ be two labeled wedges which are purely periodic of period $p$.
Suppose moreover that for every pair $(i, j)$ with $i, j \not\equiv 0 \mod p$ the label of $(i,j)$ 
in $\mathcal{W}_1$ equals the label in $\mathcal{W}_2$.
Then the finite models $\Gamma_1^F$ and $\Gamma_2^F$ are isomorphic graphs.
As a consequence, the growth rates of $\mathcal{W}_1$ and $\mathcal{W}_2$ are equal.
\end{lemma}

\begin{proof}
Let $\{i, j\} \in \Sigma_{p,0}$ an equivalence class of unordered pairs. If neither $i$ nor $j$ are divisible by $p$, then the label of $\{i, j\}$ 
is the same in $\mathcal{W}_1$ and $\mathcal{W}_2$, hence the outgoing edges from $\{i, j\}$ are the same in $\Gamma_1^F$ and $\Gamma_2^F$. Suppose on the other hand that $i \equiv 0 \mod p$ (hence, $j \not\equiv 0 \mod p$ because the pair is non-diagonal). Then, if the pair $\{p, j\}$ is non-separated, then its only outgoing edge goes to 
$\{p+1, j+1\} = \{1, j+1\}$ in $\Sigma_{p, 0}$. On the other hand, if $\{p, j\}$ is separated, then its two possible outgoing edges are 
$\{1, p+1\}$ and $\{1, j+1\}$. However, the pair $\{1, p+1\}$ is diagonal, hence no vertex in the graph has such label. Thus, independently 
of whether $\{p, j\}$ is separated, it has exactly one outgoing edge with target $\{1, j+1\}$, proving the claim. 
\end{proof}

\section{Application to core entropy} \label{S:core}

\subsection{From external angles to wedges}
We shall now apply the theory of labeled wedges to the core entropy.
As we have seen in Section \ref{S:algo}, Thurston's algorithm
allows one to compute the core entropy for periodic angles $\theta$; 
in order to interpolate between periodic angles of different periods, we shall now define for \emph{any}
angle $\theta \in \mathbb{R}/\mathbb{Z}$
a labeled wedge $\mathcal{W}_\theta$, and thus an infinite graph $\Gamma_{\theta}$
as described in the previous sections.

Recall that for each $i \geq 1$ we denote $x_i(\theta) := 2^{i-1} \theta \mod 1$, which we see as a point in $S^1$.
For each pair $(i, j)$ which belongs to $\Sigma$, we label $(i, j)$ as \emph{non-separated} if $x_i(\theta)$ and $x_j(\theta)$ 
belong to the same element of the partition $P_\theta$, or at least one of them lies on the boundary of the partition. 
If instead $x_i(\theta)$ and $x_j(\theta)$ belong to the interiors of two different intervals of $P_\theta$, then we label the 
pair $(i, j)$ as \emph{separated}. The labeled wedge just constructed will be denoted $\mathcal{W}_\theta$, and its associated 
graph $\Gamma_\theta$. The main quantity we will work with is the following:

\begin{definition}
For each external angle $\theta \in \mathbb{R}/\mathbb{Z}$, we define its \emph{growth rate} $r(\theta)$ to be the 
growth rate of the (infinite) graph $\Gamma_\theta$ associated to the labeled wedge $\mathcal{W}_\theta$ constructed above.
\end{definition}

In this section we shall prove the following theorem, which immediately implies Theorem \ref{T:main}:

\begin{theorem} \label{T:newmain}
The function $\theta \to \log r(\theta)$ is continuous, and coincides with the core entropy $h(\theta)$ for rational values of $\theta$.
\end{theorem}

In order to prove the theorem, we shall first prove that if $\theta$ is periodic, the labeled wedge $\mathcal{W}_\theta$ is periodic of the same period and pre-period.

\begin{lemma} \label{L:periodic}
If $\theta$ is periodic for the doubling map of period $p$ and pre-period $q$, then the labeled wedge $\mathcal{W}_\theta$ is periodic 
of the same period and pre-period. 
\end{lemma}

\begin{proof}
Let $\theta$ have period $p$ and pre-period $q$. 
By definition of the equivalence relation $\equiv_{p,q}$, we have $i \equiv_{p,q} j$ if and only if $2^{i-1}\theta \equiv 2^{j-1} \theta \mod 1$, 
which proves the first condition in the definition of periodic wedge. Moreover, if $(i, j)$ is a diagonal pair, then $x_i(\theta) \equiv x_j(\theta)$, 
hence the pair $(i, j)$ is non-separated, verifying the second condition.
\end{proof}

Using the results of the previous section, we are now ready to prove that the logarithm of the growth rate of $\Gamma_\theta$ coincides with the core entropy for rational angles, proving the first part of Theorem \ref{T:newmain}.

\begin{theorem}  \label{T:coincide}
Let $\theta$ a rational angle. Then the growth rate $r(\theta)$ of the infinite graph $\Gamma_\theta$ coincides with the logarithm of the core entropy: 
$$h(\theta) = \log r(\theta).$$
\end{theorem}

\begin{proof}
Let $\theta$ a rational angle, $\Gamma_\theta$ the infinite graph associated to the labeled wedge $\mathcal{W}_\theta$, 
and $\Gamma^F_\theta$ the finite model of $\Gamma_\theta$ as described in section \ref{S:finitemodel}
By unraveling the definitions, the matrix $M_\theta$ constructed in section \ref{S:algo} is exactly the adjacency matrix of the 
finite model $\Gamma^F_\theta$.
By Proposition \ref{P:finiteinfinite}, the growth rate of $\Gamma_\theta$ coincides with the growth rate of its finite model
$\Gamma^F_\theta$. Moreover, by Lemma \ref{L:PerronEig}, the growth rate of $\Gamma^F_\theta$ coincides with the largest real eigenvalue
of its adjacency matrix, that is the largest real eigenvalue of $M_\theta$. Thus, by Theorem \ref{T:algo} its logarithm is the core entropy $h(\theta)$.
\end{proof}

\subsection{Continuity of core entropy}

We shall now prove the second part of Theorem \ref{T:newmain}, namely that the function $r(\theta)$ is continuous.


\begin{proposition} \label{P:notperiodic}
For each angle $\theta \in \mathbb{R}/\mathbb{Z}$, the one-sided limits 
$$r(\theta^+) := \lim_{\theta' \to \theta^+} r(\theta'), \qquad r(\theta^-) := \lim_{\theta' \to \theta^-} r(\theta')$$ 
exist. Moreover, if $\theta$ is not purely periodic we have $r(\theta^+) = r(\theta^-) = r(\theta)$ .
\end{proposition}

\begin{proof}
As $\theta' \to \theta^+$, the labeled wedge $\mathcal{W}_{\theta'}$ stabilizes. In fact, for each $i$ consider the position of $x_i(\theta) := 2^{i-1} \theta \mod 1$
with respect to the partition $P_\theta$. If $x_i(\theta) \not \equiv \theta/2, (\theta+1)/2 \mod 1$, then, for all $\theta'$ in a neighborhood of $\theta$, $x_i(\theta')$ lies on
the same side of the partition. Otherwise, note that for any $i \in \{0,1\}$ and $j \geq 1$, 
 the functions 
$\varphi_i(\theta) := \frac{\theta}{2} + \frac{i}{2} \mod 1$ and $x_j(\theta) := 2^{j-1}\theta\mod 1$ are both continuous 
and orientation preserving but have different derivative: thus, for all $\theta' > \theta$ close enough to $\theta$, the point $x_i(\theta')$ lies on one side of the partition, 
and for $\theta' < \theta$ close enough to $\theta$ lies on the other side. Thus, the limits 
$$\mathcal{W}_{\theta^+} := \lim_{\theta' \to \theta^+} \mathcal{W}_\theta, \qquad 
\mathcal{W}_{\theta^-} := \lim_{\theta' \to \theta^-} \mathcal{W}_{\theta'}$$
 exist, and are both equal to $\mathcal{W}_{\theta'}$ if $\theta$ is not purely periodic, because then no $x_i(\theta)$ with $i \geq 1$ lies on the boundary of the 
 partition. The claim then 
follows by Lemma \ref{L:contwedge}.
\end{proof}


This proves continuity of $r(\theta)$ at all angles which are not purely periodic. We shall now deal with the purely periodic case, 
where for the moment we only know that the left-hand side and right-hand side limits of $r(\theta)$ exist.

\begin{lemma} \label{L:diffpairs}
Let $\theta$ be purely periodic of period $p$. Then $\mathcal{W}_\theta$, $\mathcal{W}_{\theta^+}$ and $\mathcal{W}_{\theta^-}$
are purely periodic of period $p$, and differ only in the labelings of pairs $(i, j)$ with either $i \equiv 0 \mod p$ or $j \equiv 0 \mod p$. 
\end{lemma}

\begin{proof}
Let $\theta$ be purely periodic of period $p$.
Note that $x_i(\theta)$ lies on the boundary of the partition $P_\theta$ if and only if $i \equiv 0 \mod p$. Thus, for all the pairs $(i, j)$ for which neither component is divisible by $p$, the label of $(i, j)$ is continuous across $\theta$, proving the second statement.
Now we shall show that if for all $\theta' > \theta$ and close enough to $\theta$ one has $x_p(\theta)$ on one side of the partition, 
then also $x_{pk}(\theta)$ is on the same side for each $k \geq 1$. This will prove that $\mathcal{W}_\theta^+$ is purely 
periodic of period $p$. The proof for $\mathcal{W}_\theta^-$ is symmetric, and $\mathcal{W}_\theta$ is purely periodic of period
$p$ by Lemma \ref{L:periodic}.
In order to prove the remaining claim, let us denote $\varphi_i(\theta) := \frac{\theta}{2} + \frac{i}{2}$, $i = 0,1$; 
since $\theta$ is purely periodic, there exists $j \in \{0,1\}$ such that $x_{pk}(\theta) = \varphi_j(\theta)$ 
for all $k \geq 1$. Now, note that for each $k \geq 1$ the derivatives satisfy the inequality 
$x_{pk}'(\theta) > \varphi'_j(\theta)$, 
thus for each $k \geq 1$ there exists $\epsilon > 0$ such that for each $\theta' \in (\theta, \theta+ \epsilon)$ one has 
$x_{pk}(\theta') > \varphi_j(\theta')$,
hence the points $x_{pk}(\theta')$ all belong to the same side of the partition independently of $k$, as required.   
\end{proof}

\begin{proposition} \label{P:bothsides}
If $\theta$ is purely periodic of period $p$, then the (infinite) graphs $\Gamma_\theta$, $\Gamma_{\theta^+}$ and $\Gamma_{\theta^-}$, 
associated respectively to $\mathcal{W}_\theta$, $\mathcal{W}_{\theta^+}$ and $\mathcal{W}_{\theta^-}$
have the same growth rate, i.e. we have the equality
$$r(\theta) = r(\theta^+) = r(\theta^-).$$
\end{proposition}

\begin{proof}
By Lemma \ref{L:diffpairs} and Lemma \ref{L:twoperiodics},  the graphs $\Gamma_\theta$, $\Gamma_{\theta^+}$ and $\Gamma_{\theta^-}$ 
have the same finite form, hence by Proposition \ref{P:finiteinfinite} they have the same growth rate, proving the claim.
\end{proof}

The Proposition thus completes the proofs of Theorem \ref{T:newmain} and Theorem \ref{T:main}.

\subsection{H\"older continuity}
By making the previous continuity arguments quantitative, we get the following 

\begin{proposition}
Let $\theta \in \mathbb{Q}/\mathbb{Z}$ such that $h(\theta) > 0$. Then there exists a neighbourhood $U$
of $\theta$, and positive constants $C, \alpha$ such that 
$$|h(\theta_1)-h(\theta_2)| \leq C |\theta_1 - \theta_2|^\alpha$$
for any $\theta_1, \theta_2 \in U$.
\end{proposition}

\begin{proof}
Let $\theta \in \mathbb{Q}/\mathbb{Z}$, and denote by $s = e^{-h(\theta)}$ the smallest real zero of $P_\theta(t)$. 
The proof is a simple quantitative version of Rouch\'e's theorem.
For simplicity, let us assume that $\theta$ is pre-periodic, and 
let $C_0$ be the minimum distance, in $S^1$, between (distinct) elements of the set 
$\{x_i(\theta)\}_{i \geq 1} \cup \{\theta/2, (\theta+1)/2\}$. 
Given another angle $\theta_1$, 
denote $s_1 = e^{-h(\theta_1)}$ the smallest real zero of $P_{\theta_1}(t)$, and choose $n$ such that 
\begin{equation} \label{e:hold0}
2^{-2n-2} C_0 \leq |\theta_1 - \theta| \leq 2^{-2n} C_0.
\end{equation}
Then, by looking at how $x_i(\theta)$ moves with respect to the partition $P_\theta$, we realize that, for $i, j \in \{1, \dots, 2n\}$, 
the pair $(i, j)$ is separated for $\theta$ if and only if it is separated for $\theta_1$.
By construction, the $n^{th}$ coefficient of $P_\theta(t)$ depends only on multi-cycles of total length $n$, 
and by Proposition \ref{P:properties} (3) all such multi-cycles live on the subgraph $\{(i, j) \ : 1 \leq i < j \leq 2 n \}$, hence 
the first $n$ coefficients of $P_\theta(t)$ and $P_{\theta_1}(t)$ coincide.  Thus we can write 
\begin{equation} \label{e:hold1}
P_{\theta_1}(s_1) - P_{\theta}(s_1) = \sum_{k = n}^\infty  c_k s_1^k \leq C_1 (s')^n
\end{equation}
for some $C_1 > 0$ and $s' \in (s_1, 1)$, where we used that $|c_k| \leq 2(2 k)^{\sqrt{2k}}$ by Proposition \ref{P:properties} (5).
Now, using $P_{\theta}(s) = P_{\theta_1}(s_1) = 0$
we have 
\begin{equation} \label{e:hold2}
P_{\theta_1}(s_1) - P_{\theta}(s_1) = P_{\theta}(s) - P_{\theta}(s_1) = (s - s_1)^k f(s_1)
\end{equation}
where $k \geq 1$ is the order of zero of $P_\theta$ at $s$ and $f(z)$ is a holomorphic function, with $f(s) \neq 0$.
By combining \eqref{e:hold1}, \eqref{e:hold2} and \eqref{e:hold0}, 
we get 
$$|s-s_1| \leq C_2 (s')^{n/k} \leq C_3 |\theta - \theta_1|^{\alpha}$$
with $\alpha  = \frac{- \log s'}{2k \log 2}$, which  implies the claim as $h(\theta) = -\log s$.
In the periodic case the proof the same, except one should argue separately for $\theta_1 < \theta$
and $\theta_1 > \theta$, and use Proposition \ref{P:bothsides}.
\end{proof}

\medskip

\noindent 
\textsc{Giulio Tiozzo} \\
\textsc{Yale University} \\
giulio.tiozzo@yale.edu


\end{document}